\RequirePackage{fix-cm}
\documentclass[smallcondensed,nospthms]{svjour3}     
\smartqed  
\usepackage{amsmath,amssymb,amsthm}
\usepackage{hyperref}
\usepackage{enumerate}
\newcommand{\ball}[2]{\mathbf{\bar B}_{#1}(#2)}

\newcommand{\gr}{\text{Gr\ }}

\newtheorem{theorem}{Theorem}[section]
\newtheorem{definition}[theorem]{Definition}
\newtheorem{proposition}[theorem]{Proposition}
\newtheorem{lemma}[theorem]{Lemma}
\newtheorem{remark}[theorem]{Remark}
\newtheorem{corollary}[theorem]{Corollary}
\newtheorem{example}[theorem]{Example}

\begin{document}

\title{Characterizations of some transversality-type properties \thanks{This
		work was supported by the Bulgarian National Scientific Fund under Grant KP-06-H22/4/04.12.2018 and
		by the Sofia	University "St. Kliment Ohridski" fund "Research
		\& Development" under contract 80-10-107 / 16.04.2020.}
}


\author{Stoyan Apostolov \and Mira Bivas \and Nadezhda Ribarska
}


\institute{S. Apostolov \at
			  Faculty of Mathematics and Informatics,  Sofia University, James Bourchier Boul. 5,  1126 Sofia, Bulgaria\\
              \email{sapostolov@fmi.uni-sofia.bg}           
           \and
           M. Bivas \at
			  Faculty of Mathematics and Informatics,  Sofia University, James Bourchier Boul. 5,  1126 Sofia, Bulgaria
			  and
			  Institute of
			  Mathematics and Informatics, Bulgarian Academy of Sciences, G.Bonchev str., bl. 8, 1113 Sofia, Bulgaria\\
			  \email{mira.bivas@math.bas.bg}
		  \and
		  N. Ribarska \at
		  Faculty of Mathematics and Informatics,  Sofia University, James Bourchier Boul. 5,  1126 Sofia, Bulgaria
		  and
		  Institute of
		  Mathematics and Informatics, Bulgarian Academy of Sciences, G.Bonchev str., bl. 8, 1113 Sofia, Bulgaria\\
		  \email{ribarska@fmi.uni-sofia.bg}
}

\date{Received: date / Accepted: date}

\maketitle

\begin{abstract}
Primal characterizations (necessary and sufficient conditions) and slope characterizations of subtransversality, intrinsic transversality and transversality are obtained. The metric nature of intrinsic transversality is established. The relation of intrinsic transversality and tangential transversality is clarified. The equivalence of our characterization of intrinsic transversality and the primal characterization of intrinsic transversality of Thao et al. in the setting of Hilbert spaces is proved, while in general Banach spaces our characterization is less restrictive.

\end{abstract}

\begin{acknowledgements}
The authors are grateful to Prof. A. Dontchev for his useful comments and suggestions.
\end{acknowledgements}

\section{Introduction}

	Transversality is a classical concept of mathematical analysis and differential topology. Recently, it has proven to be useful in
	variational analysis as well. As it is stated in \cite{Ioffe}, the transversality-oriented language is extremely natural and convenient
	in some parts of variational analysis, including subdifferential calculus and nonsmooth optimization, as well as in proving sufficient
	conditions for linear convergence of the alternating projections algorithm (cf. \cite{DIL2015}).
	
	There are various transversality-type properties reflecting the various needs of the possible applications. In the literature there exist many notions generalizing the classical transversality as well as transversality of cones. Some of them are introduced under different names by different authors, but actually coincide. We refer to \cite{Kruger2018} for a survey of terminology and comparison of the available concepts. The central ones among them are \textit{transversality} and \textit{subtransversality}. They are also objects of study in the recent book \cite{IoffeBook}. One of the reasons for that is the close relation to metric regularity and metric subregularity, respectively. Recently the notion of intrinsic transversality which is intermediate between subtransversality and transversality has been introduced and steadily grows in importance. These notions (which some authors call ``good arrangements of sets'') have been studied in details. See , e.g., \cite{nese},\cite{suff}, \cite{dualsuff}, \cite{Kruger2017}, and the literature therein.

One of the starting points of this investigation was a question of A.Ioffe about finding a metric characterization of intrinsic transversality. In fact, a variety of characterizations of intrinsic transversality in various settings are known (Euclidean, Hilbert, Asplund, Banach and normed linear
spaces) but all of them involve the linear structure of the space. The reason is that researchers are mainly concentrated on the dual space. To the best of our knowledge, the first primal characterization of intrinsic transversality is obtained in \cite{TBCV2018} where the structure of a Hilbert space is assumed in most of the considerations.
	
We arrived to the study of transversality of sets when investigating Pontryagin's type maximum principle for  optimal control problems with terminal
constraints in infinite dimensional state space. In \cite{BKRtt} we formulated necessary optimality conditions for optimization problems in
terms of abstract Lagrange multipliers and established intersection rules for tangent cones
in Banach spaces, making use of another transversality-type property which we called tangential transversality. Many questions about this tool remained open (see \cite{BKRtt}, p. 28).
These questions were the other starting point of our investigation.

The result of our study was somewhat surprising: it happened that intrinsic transversality and tangential transversality are ``almost'' equivalent. Moreover, the relation is very easy to establish, given the characterization of intrinsic transversality via the slope of coupling function due to Ioffe and Lewis. Thus a primal space characterization of intrinsic transversality has been obtained. We put a significant effort in clarifying the exact relationship of this characterization and the primal characterization of intrinsic transversality obtained by Thao et al. in \cite{TBCV2018}, which they call property ($\cal P$). We proved that property ($\cal P$) implies our characterization in general Banach space setting and these properties are equivalent in Hilbert space setting. We would like to emphasize that the property we introduce is simpler (or at least it looks simpler) than the property ($\cal P$) -- less variables are involved.

Establishing the exact relationship between intrinsic transversality and tangential transversality helped us to obtain primal space infinitesimal characterizations and slope characterizations of both transversality and subtransversality close in nature to tangential transversality. These characterizations make it obvious that
$$\mbox{transversality} \implies
	\parbox{6em}{\begin{center}tangential\\transversality\end{center}}  \implies \parbox{6em}{\begin{center}
		intrinsic\\transversality
		\end{center}} \implies  \mbox{subtransversality} \, .$$
and neither implication is invertible. This hierarchy of the properties and of their respective slope characterizations sheds new light on the topic. There have been known primal sufficient conditions and primal necessary conditions for transversality and subtransversality, but no primal characterizations (see \cite{suff} and \cite{nese}). The relationship of our characterization to these conditions is very similar to the relationship of our characterization of intrinsic transversality to property ($\cal P$) -- we work with less points which makes the situation simpler.

	The paper is organized as follows. The second section contains some necessary definitions and known assertions. In the third section a primal characterization of subtransversality is obtained. To do it,  we prove a technical result (Lemma \ref{ttth}) allowing to pass from a local inequality to a global one. Its proof essentially appeared in \cite{BKRtt} and it is based on transfinite induction. It could have been proved using Ekeland's variational principle like most ``rate of descend'' results in the literature, but we prefer the transfinite induction because it is really
	direct to employ and requires very little thought -- a simple induction enables the transition to a global property from a local one in a
	straightforward manner. 
	In our understanding this kind of argument is natural and saves one from the necessity of seeking
	for the ``right'' function in every particular case. Moreover, a slope characterization of subtransversality is proved in the third section. Section 4 is devoted to transversality. A primal characterization of transversality is obtained from the respective characterization of subtransversality. The relation to tangential transversality is clarified emphasizing the fact that in the ``uniform'' situation of the transversality property the existing of ``a positive step'' and the existing of ``an interval $(0,\delta)$ of possible steps'' are equivalent. Using this, two slope characterizations of transversality are obtained. Section 5 deals with intrinsic transversality. A primal characterization (of purely metric nature) and a slope characterization are readily proved. The exact relation of our approach and the approach of Thao et al. in \cite{TBCV2018} is established. This is the most technical result in the paper. We obtain the coincidence of subtransversality and intrinsic transversality in the convex case as an easy consequence.

\section{Preliminaries}
Throughout the paper if $(X,d)$ is a metric space, $\mathbf{B}_{r}(x_0)$ will denote the open ball centered at $x_0$ with radius $r$:  $\mathbf{B}_{r}(x_0):=\{x\in Y\ | \ d(x,x_0)<r\}$. The closed ball will be denoted by $\mathbf{\bar B}_{r}(x_0)$.

In what follows, for given metric spaces $X$ and $Y$, $F$ is called a set-valued map between $X$ and $Y$, denoted by $F:X\rightrightarrows Y$, if $F:X\to 2^Y$. The graph of $F$, denoted by $\gr F$, is defined by 
$$\gr F:=\{(x,y)\in X\times Y\ | \ y\in F(x)\}.$$ The inverse map of $F$, $F^{-1}:Y\rightrightarrows X$ is defined by 
$$ F^{-1}(y):=\{x\in X\ | \ y\in F(x)\},\ \mbox{whenever}\ y\in Y.$$

For a set $A$ in a metric space $X$, we denote by $\delta_A : X\to\mathbb{R}\cup\{+\infty\}$ its indicator function
$$\delta_A(x) = \begin{cases}
0, &\quad\text{if } x\in A\\
+\infty, &\quad\text{otherwise.} \\ 
\end{cases}$$ 

Following \cite{GMT1980} we introduce two types of slopes.
\begin{definition}
	Consider a metric space $X$, a function $f : X \to \mathbb{R}\cup \{\pm\infty\}$ and a point $\bar x \in X$ such that
	$f(\bar x)$ is finite. The slope of $f$ at $\bar x$ is
	$$|\nabla f|(\bar x) := \limsup_{x\to \bar x}\frac{\max\{f(\bar x)-f(x), 0\}}{d(\bar x, x)} \, .$$
	The nonlocal slope is 
	$$|\nabla f|^{\diamond}(\bar x) := \sup_{x\neq \bar x}\frac{\max\{f(\bar x)-f(x), 0\}}{d(\bar x, x)} \, .$$
\end{definition}

\begin{definition}\label{coupling}
For subsets $A$ and $B$ of the metric space $X$, the so-called ``coupling function'' $\phi:X\times X\to \mathbb{R}\cup\{+\infty\}$
is  defined as
$$
    \phi(x,y)=\delta_A(x)+d(x,y)+\delta_B(y) \, .
$$
\end{definition}

The above definition has been introduced in \cite{DIL2014}.

We endow the Cartesian product $X\times Y$ of the metric spaces $(X,d_X)$ and $(Y,d_Y)$, with the metric 
$d((x_1,y_1),(x_2,y_2))=d_X(x_1,x_2)+d_Y(y_1,y_2)$ for the sake of simplicity. The particular choice of the metric is relevant only to the constants involved. 
However, our goal in this paper is to derive qualitative results, so that we are not concerned with the constants. 

 We remind the already classical definitions.

\begin{definition}
	Let $X$ and $Y$ be metric spaces, $F:X\rightrightarrows Y$ and
	$(\bar x,\bar y)\in \gr F$. We say that $F$ is (metrically) regular at $(\bar x,\bar y)$ if there exist $K>0$ and $\delta>0$ such that for all $x\in {\mathbf{B}}_\delta(\bar x)$ and all $y\in {\mathbf{B}}_\delta(\bar y)$ the following inequality holds:
	$$ d(x,F^{-1}(y))\le Kd(y,F(x)).$$

\end{definition}

\begin{definition}
	Let $X$ and $Y$ be metric spaces, $F:X\rightrightarrows Y$ and
	$(\bar x,\bar y)\in \gr F$. We say that $F$ is (metrically) subregular at $(\bar x,\bar y)$ if there exist $K>0$ and $\delta>0$ such that for all $x\in {\mathbf{B}}_\delta(\bar x)$ the following inequality holds:
	$$ d(x,F^{-1}(\bar y))\le Kd(\bar y,F(x)).$$
	
\end{definition}
Assume that $A$ and $B$ are subsets of the normed space $X$. Consider the function $H_{A,B}:X\times X\to X$ defined as

\begin{equation}\label{hmap}
H_{A,B}(x_1,x_2)=\begin{cases}
\{x_1-x_2\}, & x_1\in A,\ x_2\in B \\
\varnothing , & else 
\end{cases}
\end{equation}

\begin{definition}\label{def_tran}
	Let $X$ be a Banach space, and $A,\ B$ be closed subsets of $X$. Let $\bar x\in A\cap B$. Then $A$ and $B$ are called  transversal at $\bar x$ if $H_{A,B}$ is regular at $((\bar x,\bar x),\mathbf{0})$. 
	
\end{definition}

\begin{definition}\label{def_subtr}
	Let $X$ be a Banach space, and $A,\ B$ be closed subsets of $X$. Let $\bar x\in A\cap B$. Then $A$ and $B$ are called  subtransversal at $\bar x$ if $H_{A,B}$ is subregular at $((\bar x,\bar x),\mathbf{0})$. 
	
\end{definition}

A characterization of transversality derived in \cite{Ioffe2000} (cf. \cite{Kruger2006})
is 
\begin{proposition}\label{tra}
	Let $A$ and $B$ be closed subsets of the normed space $X$.
	$A$ and $B$ are transversal at $\bar x \in A \cap B$, if and only if there exists    $K>0$ and $\delta>0$ such that
	$$d(x, (A-a)\cap (B-b)) \leq K(d(x, A-a) + d(x, B-b)) $$
	for all $x\in \ball{\delta}{\bar x}$ and $a,b\in\ball{\delta}{\mathbf{0}}$.
\end{proposition}
One observes that only one of the sets could be translated, i.e. we may take $a=0$ and only vary $b$.

When $a$ and $b$ are fixed to be $\mathbf{0}$ in the last definition, a similar characterization of subtransversality is obtained (cf. \cite{Ioffe2000}): 
\begin{proposition}\label{subtrineq}
	Let $A$ and $B$ be closed subsets of the complete metric space $X$.
	$A$ and $B$ are subtransversal at $\bar x \in A \cap B$, if and only if there exists    $K>0$ and $\delta>0$ such that
	$$d(x, A\cap B) \leq K(d(x, A) + d(x, B)) $$
	for all $x\in \mathbf{B}_{\delta}(\bar x)$.
\end{proposition}
Thus we observe that $A$ and $B$ are transversal at $\bar x \in A\cap B$ if and only if the subtransversality inequality holds for 
$A-a$ and $B-b$ with constant $K$ for all $x\in\ball{\delta}{\bar x}$ and $a,b\in\ball{\delta}{\mathbf{0}}$.

It is worth noting that while the definitions of transversality and subtransversality clearly make use of the linear structure, the
characterization of subtransversality, given by Proposition \ref{subtrineq}, is purely metric. Thus, one may think of subtransversality as of metric concept. However, 
all chracterizations of transversality use the linear structure. 

\section{Primal space characterizations of subtransversality}

In this section we obtain primal space characterizations of subtransversality. In the papers \cite{nese} and \cite{suff} (see Remark 3.5 in \cite{nese}) similar conditions are presented. It is proved that these conditions are characterizations (both necessary and sufficient) only in the convex case.

Our approach is to some extend motivated by the considerations in the paper \cite{BKRtt}. In it, the notion of tangential transversality is 
introduced as a sufficient condition for nonseparation of sets, tangential intersection properties and a Lagrange multiplier rule. The corresponding definition follows.

\begin{definition}\label{tangtr}
	Let $A$ and $B$ be closed subsets of the metric space $X$.
	We say that $A$ and $B$ are tangentially transversal at $\bar x \in A \cap B$, if there exist $M>0$, $\delta >0$ and  $\eta >0$ such that for any
	two different points $x^A \in   {\overline{\mathbf{B}}}_\delta(\bar x)\cap A$ and $x^B \in   {\overline{\mathbf{B}}}_\delta(\bar x)\cap B$, there exist sequences $t_m\searrow 0$,  $\{x_m^A\}_{m\ge 1} \subset A$ and $\{x_m^B\}_{m\ge 1} \subset B$ such that for all $m$ $$d(x_m^A, x^A)\le t_mM,\ \ d(x_m^B, x^B)\le t_mM,\ \ d(x_m^A,x_m^B) \le d(x^A,x^B) - t_m \eta\, .$$
\end{definition}
Clearly, the three constants $M,\delta, \eta$ are redundant. More specifically, one can choose $M=1$, which changes $\eta$, or choose $\eta=1$ and change $M.$

The above definition can be reformulated equivalently.

\begin{proposition}\label{tangtr_eq}
	Let $A$ and $B$ be closed subsets of the metric space $X$.
	$A$ and $B$ are tangentially transversal at $\bar x \in A \cap B$, if and only if there exist $\delta >0$ and  $\zeta >0$ such that for any
	two different points $x^A \in   {\overline{\mathbf{B}}}_\delta(\bar x)\cap A$ and $x^B \in
	{\overline{\mathbf{B}}}_\delta(\bar x)\cap B$, there exist sequences   $\{x_m^A\}_{m\ge 1} \subset A$ and 
	$\{x_m^B\}_{m\ge 1} \subset B$  converging to $x^A$ and $x^B$ repsectively, and such that for all $m$ 
	$$ d(x_m^A,x_m^B) \le d(x^A,x^B) - \zeta\max\{d(x_m^A, x^A), d(x_m^B, x^B)\}$$
	and $\max\{d(x_m^A, x^A), d(x_m^B, x^B)\}>0.$
\end{proposition}

Now we introduce a weaker notion. Note that the main difference is that ``there exists a sequence $\{ t_n\}_{n=1}^\infty$  of positive reals tending to zero such that for every $t_n$ belonging to it $\dots$" is replaced by 
``there exists a  positive real $\theta$  such that $\dots$". This is indeed a significant difference, as it will be shown later on. The other weakening in the definition, ``$\bar x \in A \cap B$'' to ``$A\cap \ball{\frac{\delta}{2(1+2M)}}{\bar x}\ne \varnothing $, 
	$B\cap \ball{\frac{\delta}{2(1+2M)}}{\bar x}\ne \varnothing $'', is for purely technical reasons.
\begin{definition}\label{TTrevisited}
	Let $A$ and $B$ be closed subsets of the metric space $X$
	and $\bar x \in X$. We say that $A$ and $B$ have \textit{property} $(\mathcal{T})$ at $\bar x$ if there exist
	$\delta > 0$ and $M > 0$ such that $A\cap \ball{\frac{\delta}{2(1+2M)}}{\bar x}\ne \varnothing $, 
	$B\cap \ball{\frac{\delta}{2(1+2M)}}{\bar x}\ne \varnothing $ and for any $x^A \in A \cap   \mathbf{\bar B}_\delta
	(\bar x)$ and $x^B \in B \cap    \mathbf{\bar B}_\delta
	(\bar x)$ with $x^A\not =x^B$ there exist $\theta  > 0$, $\hat x^A \in A$ and  $\hat x^B \in B$ such that
	$$d(x^A, \hat x^A)\le \theta M \ , \ \
	d(x^B, \hat x^B)\le \theta M \ \mbox{ and } \
	d(\hat x^A, \hat x^B)\le d(x^A,  x^B)  - \theta  \ .$$
	Equivalently, $A$ and $B$ have $(\mathcal{T})$ at $\bar x$ if and only if there exist
	$\delta > 0$ and $M > 0$ such that $A\cap \ball{\frac{\delta}{2(1+2M)}}{\bar x}\ne \varnothing $ , 
	$B\cap \ball{\frac{\delta}{2(1+2M)}}{\bar x}\ne \varnothing $
	and for any $x^A \in A \cap   \mathbf{\bar B}_\delta
	(\bar x)$ and $x^B \in B \cap    \mathbf{\bar B}_\delta
	(\bar x)$ with $x^A\not =x^B$ there exist $\hat x^A \in A$ and  $\hat x^B \in B$ such that
	$$d(\hat x^A, \hat x^B)\le d(x^A,  x^B)  - \frac{1}{M} \max\{d(x^A, \hat x^A), d(x^B, \hat x^B) \}  $$
	and $\max\{d(x^A, \hat x^A), d(x^B, \hat x^B) \}>0.$
\end{definition}

Note that in this definition we do not require the point $\bar x$ to be in the intersection of $A$ and $B$, 
only to be sufficiently close to both $A$ and $B$.

The lemma below is the main technical result, whose direct corollaries will justify the benefits of the above definition. Its proof may seem long, because it essentially contains the proof of the Ekeland variational principle. We prefer to prove it using transfinite induction, instead of some of the decreasing principles in the literature, in order to emphasize the usefulness of the method and its geometrical intuition.  

\begin{lemma}\label{ttth}
    Let $A$ and $B$ be closed subsets of the complete metric space $(X,d)$ and $\bar x\in X$. Let  $A$ and $B$ have \textit{property $(\mathcal{T})$} at $\bar x$ with constants $\delta$ and $M$.  Let  $x^A\in A$ with $\displaystyle d(x^A,\bar x)\le \frac{\delta}{1+2M}$ and $ x^B\in B$ with $\displaystyle d(x^B,\bar x)\le \frac{\delta}{1+2M}$.    Then, there exists $x^{AB} \in A \cap B$ with
    $$d(x^{AB}, x^A) \le  {M}  d(x^A,x^B) \mbox{ and } d(x^{AB}, x^B) \le  {M}  d(x^A,x^B) \, .$$
\end{lemma}

\begin{proof}

   If the points $x^A$ and $x^B$ coincide, the assertion of the theorem is trivial. If $d(x^A,x^B) >0$, we are going to construct inductively three transfinite sequences
    indexed by ordinal numbers (cf., for example,  $\S$ 2 Ordinal
    numbers of Chapter 1 in \cite{Jech}). More precisely, we prove that
    there exist an ordinal number $\alpha_0$ and transfinite sequences
    $\{x_\alpha^A\}_{1 \le \alpha \le \alpha_0}\subset X$,
    $\{x_\alpha^B\}_{1 \le \alpha \le \alpha_0}\subset X$,
    $\{t_\alpha\}_{1 \le \alpha \le \alpha_0}\subset [0, +\infty)$,
    such that $x_{\alpha_0}^A =x_{\alpha_0}^B$ and  for each $\alpha \in [1, \alpha_0 ]$
    we have that the
    following properties hold true:

    \begin{enumerate} 
    \item [(S0)] $\{x_\alpha^A\}\in \ball{\delta}{\bar x}\cap A$ and $\{x_\alpha^B\}\in \ball{\delta}{\bar x}\cap B$;\\
    \item [(S1)]  $d(x_\alpha^A , x_\alpha^B) \le d(x^A , x^B)  - t_\alpha$ (and hence $t_\alpha$  is bounded by $d(x^A , x^B)$);\\
        \item [(S2)]  $ d(x_\alpha^A  , \bar x) \le d(x^A , \bar x)+t_\alpha M$ and $ d(x_\alpha^B , \bar x) \le d(x^B  ,  \bar x)+t_\alpha M$;\\
        \item [(S3)]  $d(x_\alpha^A, x_\gamma^A) \le M \left(t_\alpha - t_\gamma\right)$ and $d(x_\alpha^B,x_\gamma^B) \le M \left(t_\alpha - t_\gamma\right)$  for each $\gamma\le \alpha$.\\
    \end{enumerate}

    We implement our construction using induction on $\alpha$. The process terminates when $x_\alpha^A=x_\alpha^B$ for some $\alpha$, and this $\alpha$ is named $\alpha_0$.
    We start with  $x_1^A := x^A\in  \mathbf{\bar B}_\delta (\bar x)\cap A$, $x_1^B :=x^B\in \mathbf{\bar B}_\delta (\bar x)\cap B$ and $t_1= 0$.
    It is straightforward to verify     the inductive
    assumptions (S1)-(S3)   for $\alpha = 1$.

    Assume that $x_\beta^A\in \mathbf{\bar B}_\delta (\bar x)\cap A$, $x_\beta^B\in \mathbf{\bar B}_\delta (\bar x)\cap B$ and
    $t_\beta$ are constructed and (S1)-(S3) are true for all ordinals $\beta$  less than $\alpha$ and the process has not been terminated.

    Let us first consider the case when $\alpha$  is a successor ordinal, i.e. $\alpha = \beta+ 1$. As $\beta <\alpha_0$
    (the process has not been terminated), we have $d(x_\beta^A ,  x_\beta^B )\ne 0$. Moreover $(S0)$ holds, so we can apply property $(\mathcal{T})$ to obtain $\theta>0,\hat x_\beta^A,\hat x_\beta^B$, and we define $t_\alpha:=t_\beta+\theta$, $x_\alpha^A:=\hat x_\beta^A$ and $x_\alpha^B:=\hat x_\beta^B$. 
    Now we have $x_\alpha^A \in A$, $x_\alpha^B\in B$, $d(x_\alpha^A, x_\beta^A)\le M \theta$,  $d(x_\alpha^B, x_\beta^B)\le M \theta$
    and $d(x_\alpha^A ,x_\alpha^B )\le d(x_\beta^A,x_\beta^B) - \theta$.
    Using the inductive assumption, we have
    $$d(x_\alpha^A ,x_\alpha^B )\le d(x_\beta^A,x_\beta^B) -\theta \le d(x^A,x^B) - t_\beta - \theta= d(x^A,x^B) - t_\alpha \ .$$
    Therefore, (S1) is verified for $\alpha$.

    Now the inequalities $d(x_\alpha^A, x_\beta^A)\le M \theta$,  $d(x_\alpha^B, x_\beta^B)\le M \theta$ and the inductive assumption (S2) for $\beta$ yield
$$d(x_\alpha^A , \bar x) \le   d(x_\beta^A , \bar x) +d(x_\alpha^A, x_\beta^A) \le d(x^A , \bar x)+t_\beta M + M \theta =  d(x^A , \bar x)+t_\alpha M ,$$
$$d(x_\alpha^B , \bar x) \le   d(x_\beta^B , \bar x) +d(x_\alpha^A, x_\beta^B) \le d(x^B , \bar x)+t_\beta M + M \theta =  d(x^B , \bar x)+t_\alpha M .$$

    Thus (S2) is verified for $\alpha$. Using the estimate $t_\beta \le d(x^A , x^B)$ from (S1) for $\alpha$, the assumption of the theorem  and the above inequalities, we obtain
   $$d(x_\alpha^A , \bar x) \le  d(x^A , \bar x)+t_\alpha M \le d(x^A,\bar x)  +  {M} d(x^A,x^B)\le$$
   $$d(x^A,\bar x)+M(d(x^A,\bar x)+d(x^B,\bar x))\le\frac{\delta}{1+2M}+M\frac{2\delta}{1+2M}=\delta $$
   which means that $x_\alpha^A\in \mathbf{\bar B}_\delta (\bar x)$. Similarly $x_\alpha^B\in \mathbf{\bar B}_\delta (\bar x)$. Thus (S0) holds.

    Now  let $\gamma<\alpha$. Then
$$
    d(x_\alpha^A , x_\gamma^A) \le d(x_\beta^A,x_\gamma^A)+ d(x_\alpha^A, x_\beta^A) \le 
    M \left(t_\beta - t_\gamma\right) + M \theta = M \left(t_\alpha - t_\gamma\right)
$$
    and in the same way
    $$
    d(x_\alpha^B , x_\gamma^B) \le d(x_\beta^B,x_\gamma^B)+ d(x_\alpha^B, x_\beta^B) \le
    M \left(t_\beta - t_\gamma\right) + M \theta = M \left(t_\alpha - t_\gamma\right) \ .
$$

    We have verified the inductive assumptions (S0)-(S3) for the case of a successor ordinal $\alpha$.

    We next consider the case when $\alpha$ is a limit  ordinal
    number. Let $\beta<\alpha$ be arbitrary. Then $\beta+1<\alpha$
    too. Since the transfinite process has not stopped at   $\beta+1$,
    then $d(x_\beta^A , x_\beta^B)>0$, and hence taking into account
    (S1) we obtain that
    $t_\beta < d(x_\beta^A , x_\beta^B)$.
    Hence the increasing transfinite sequence
    $\{t_\beta\}_{1\le\beta<\alpha}$ is bounded, and so it is
    convergent. We denote $t_\alpha:=\lim_{\beta \to \alpha} t_\beta$. Since $d(x_\beta^A , x_\gamma^A ) \le
    (t_\beta-t_\gamma)M$, the transfinite sequence
    $\{x_\beta^A\}_{1\le\beta<\alpha}$ is fundamental. Hence there
    exists $x_\alpha^A$ so that $\{x_\beta^A\}_{1\le\beta<\alpha}$ tends
    to $x_\alpha^A$ as $\beta$ tends to $ \alpha$ with $\beta<\alpha$. In the same way one can
    prove the existence of $x_\alpha^B$ so that the transfinite
    sequence $\{x_\beta^B\}_{1\le\beta<\alpha}$ tends to $x_\alpha^B$ as
    $\beta$ tends to $\alpha$. To verify the inductive assumptions (S1)-(S3) for
    $\alpha$, one can just take a limit for $\beta$ tending to $\alpha$ with $\beta<\alpha$ in
    the same assumptions written for each  $\beta <\alpha$. For (S0) one uses that $A$ and $B$ are closed.

    We have   constructed inductively the transfinite sequences $\{ x_\beta^A \}_{\beta \le \alpha}\subset A$, $\{ x_\beta^B \}_{\beta \le \alpha}\subset B$ and $\{ t_\beta\}_{\beta
        \le \alpha}\subset [0,+\infty)$. The process will terminate when
    $x_\alpha^A=x_\alpha^B$ for some $\alpha$. Since
    $d(x_\alpha^A , x_\alpha^B) \le d(x^A , x^B)  - t_\alpha$
    and the transfinite sequence $t_\alpha$ is strictly increasing, the  equality
    $x_\alpha^A=x_\alpha^B$ will be satisfied for some
    $\alpha=\alpha_0$ strictly preceding the first uncountable ordinal
    number. Indeed, the successor ordinals indexing the
    so constructed transfinite sequences form a countable set (because to every successor
    ordinal $\alpha+1$ corresponds the open interval $(t_\alpha, t_{\alpha+1})\subset \mathbb{R}$,
    these intervals are disjoint and the rational numbers are countably many and
    dense in $\mathbb{R}$). Therefore, $\alpha_0$ is countably accessible.
    On the other hand-side, assuming the Axiom of countable choice, $\omega_1$ is not countably accessible).
    Hence our inductive process terminates before $\omega_1$.

    Then we put $x^{AB}:=x_{\alpha_0}^A=x_{\alpha_0}^B \in A \cap B$ and because of (S1)
    we have that
    $t_{\alpha_0} \le d(x^A , x^B)$. 
    Applying (S3)  we obtain
    $d(x^{AB}, x^A) \leq M(t_{\alpha_0}-t_1) \le M d(x^A,x^B)$ hence $ \ d(x^{AB}, x^B)  \le M d(x^A,x^B) \, . $

    This completes the proof.
\end{proof}

Completeness is crucial in the above lemma. The following theorem is formulated in a way that enables us to use it to obtain primal space characterizations both for subtransversality and transversality.

\begin{theorem}\label{tansub}
	Let $A$ and $B$ be closed subsets of the complete metric space $X$ and $\bar x\in X$. If $A$ and $B$ have property $(\mathcal{T})$ at 
	$\bar x$, then there exist $K>0$ and $\delta>0$ such that
	\begin{equation}\label{ineq_subtr}
	d(x, A\cap B) \leq K(d(x, A) + d(x, B))
	\end{equation}
	for all $x\in \mathbf{B}_{\delta}(\bar x)$. 
	
	If there exist $K>0$ and $\delta>0$ such that \eqref{ineq_subtr} holds for all $x\in \mathbf{B}_{\delta}(\bar x)$, 
	$A\cap \ball{\frac{\delta}{4K+10}}{\bar x}\ne \varnothing $ and
	$B\cap \ball{\frac{\delta}{4K+10}}{\bar x}\ne \varnothing $, then $A$ and $B$ 
	have property $(\mathcal{T})$ at $\bar x$.

\end{theorem}

\begin{proof} Let $A$ and $B$ have property $(\mathcal{T})$ with constants $M,\ \delta$. Let $\displaystyle \hat \delta:=\frac{\delta}{8(1+2M)}.$
	Let $x\in \ball{\hat\delta}{\bar x}$ and choose $\varepsilon\in (0,\hat \delta)$. Then there exists $x^A\in A$, such that $d(x,x^A)\le d(x,A)+\varepsilon$. 
	We have that $d(x,A)\le d(x, \bar x) + d(\bar x, A) \le \hat \delta +  \frac{\delta}{2(1+2M)} \le 5 \hat\delta$, 
	so that $d(x,x^A)\le 6\hat\delta$. Since $d(x,\bar x)\le \hat\delta$, the triangle inequality implies 
	$$d(x^A,\bar x)\le 7\hat\delta<\frac{\delta}{1+2M}.$$
	Similarly, we find $x^B\in B$, such that $d(x,x^B)\le d(x,B)+\varepsilon$ and 
	$$d(x^B,\bar x)<\frac{\delta}{1+2M}.$$
	Then $x^A$ and $x^B$ satisfy the requirements in Lemma \ref{ttth}. Hence, there is $x^{AB}\in A\cap B$, such that 
	$$d(x^{AB}, x^A) \le  M  d(x^A,x^B) \mbox{ and } d(x^{AB}, x^B) \le  M  d(x^A,x^B) \, .$$
	We estimate
	\begin{align*}
	d(x, &A\cap B)\le d(x, x^{AB})\le d(x,x^A)+d(x^A,x^{AB})\le d(x,A)+\varepsilon+Md(x^A,x^B)\\
	&\le d(x,A)+\varepsilon+ M(d(x,x^A)+d(x,x^B))\\
	&\le d(x,A)+\varepsilon+M(d(x,A)+\varepsilon+d(x,B)+\varepsilon)\\
	&\le(M+1)(d(x,A)+d(x,B))+\varepsilon(1+2M)
	\end{align*}
	Letting $\varepsilon\to 0$ proves \eqref{ineq_subtr} with constants $\hat\delta$ and $M+1$.
	
	For the second part, let \eqref{ineq_subtr} hold with constants $\delta $ and $K$. Take $x^A\in A\cap \ball{\delta}{\bar x}$ and $x^B\in B\cap \ball{\delta}{\bar x}$ and for $\varepsilon := d(x^A,x^B) >0$, find  $x^{AB}\in A\cap B$ such that 
	$$d(x^A,x^{AB})< d(x^A,A\cap B)+\varepsilon \le K d(x^A, B)+\varepsilon $$ 
	$$\le K d(x^A,x^B) +\varepsilon =(K+1) d(x^A,x^B) .$$ 
	
	Then
	$$d(x^B,x^{AB}) \le d(x^A,x^B) + d(x^A,x^{AB}) $$
	$$\le d(x^A,x^B)  + (K+1) d(x^A,x^B) =(K+2) d(x^A,x^B) .$$
	Now property $(\mathcal{T})$ follows with
	$\hat x^A = \hat x^B = x^{AB}$, $\theta = d(x^A,x^B)>0$ and $M=K+2$, because proximity of $A$ and $B$ to $\bar x$ is assumed. 
	
\end{proof}
As a corollary we obtain that property $(\mathcal{T})$ is an equivalent characterization of subtransversality in the presence of completeness.
\begin{corollary}\label{corol1}
    If $\bar x \in A\cap B$, where $A$ and $B$ are closed subsets of the complete metric space $X$, then $A$ and $B$ have property $(\mathcal{T})$ at $\bar x$
	if and only if $A$ and $B$ are subtransversal at $\bar x$. 
\end{corollary}

The following proposition is a reformulation of Corollary \ref{corol1}.

\begin{proposition}\label{prop_coupsubtr}
	Under completeness of the space $X$, $A$ and $B$ are subtransversal at $\bar x$ if and only if there exist $\delta>0$ and $\kappa>0$ such that for all 
	$x\in A\cap \ball{\delta}{\bar x}$ and $y\in B\cap \ball{\delta}{\bar x}$, $x\ne y$, it holds
	$$|\nabla \phi|^\diamond (x, y) = \sup_{(u,v)\ne (x,y)}\frac{\max\{\phi(x,y)-\phi(u,v),0\}}{d((x,y),(u,v))}\ge \kappa.$$ 
\end{proposition}
\begin{proof}
Assume that the sets $A$ and $B$ are subtransversal. Then according to Corollary \ref{corol1} property $(\mathcal{T})$ holds.

Then, there exist $M>0$, $\delta >0$ such that for any
	two different points $x^A \in \ball{\delta}{\bar x}\cap A$ and $x^B \in \ball{\delta}{\bar x}\cap B$, 
	there exist $\theta>0$, $x^A \in A$ 
	with $d(x^A,\hat x^A)\le  \theta M$, and $x^B \in B$ with $d(x^B,\hat x^B)\le  M\theta$,  
	and the following inequality holds true $$d(\hat x^A,\hat x^B) \le d(x^A,x^B) -\theta .$$
	Clearly, the last inequality yields that $(\hat x^A,\hat x^B)\ne (x^A,x^B).$
Remind that $\phi(x^A,x^B)=d(x^A,x^B)$ (Definition \ref{coupling}) since $x^A\in A$ and $x^B\in B$ and similarly $\phi(\hat x^A,\hat x^B)=d(\hat x^A,\hat x^B).$
This leads to
	$$d(x^A,x^B)-d(\hat x^A,\hat x^B) \ge \theta\ge \frac{d(x^A,\hat x^A)+d( x^B, \hat x^B)}{2M}$$
	Therefore
	$$\frac{\phi(x^A , x^B) - \phi(\hat x^A , \hat x^B)}{d((x^A, x^B),(\hat x^A, \hat x^B))}=\frac{d(x^A , x^B) - d(\hat x^A , \hat x^B)}{d(x^A,\hat x^A)+d(x^B, \hat x^B)} \ge \frac{1}{2M}\, .$$
	Thus we obtain that 
	$$ |\nabla \phi|^\diamond(x^A, x^B) \ge \frac{1}{2M}$$ for
	any
	two different points $x^A \in \ball{\delta}{\bar x}\cap A$ and $x^B \in \ball{\delta}{\bar x}\cap B$. 

	For the reverse direction, we have that for some $\delta>0$ and $\kappa>0$ and for any two different points $x\in A\cap \ball{\delta}{\bar x}$ and $y\in B\cap \ball{\delta}{\bar x}$ $$|\nabla \phi|^\diamond (x, y) = \sup_{(u,v)\ne (x,y)}\frac{\max\{\phi(x,y)-\phi(u,v),0\}}{d((x,y),(u,v))}\ge \kappa>0.$$
	So fix $x\in A\cap \ball{\delta}{\bar x}$ and $y\in B\cap \ball{\delta}{\bar x}$ with $x\ne y$. We obtain that there are $u$ and $v$ such that 
	$$\frac{\phi(x,y)-\phi(u,v)}{d(x,u)+d(y,v)}\ge\frac{\kappa}{2}$$
As above, $\phi(x,y)=d(x,y)$. Observe that $\phi(u,v)<\infty$, thus $u\in A$ and $v\in B$.
Hence $d(u,v)\le d(x,y)-\theta$ where $\displaystyle\theta=\frac{\kappa}{2}(d(x,u)+d(y,v))>0.$
	Moreover $\displaystyle d(x,u)\le \frac{2}{\kappa}\theta$ and $\displaystyle d(y,v)\le \frac{2}{\kappa}\theta$.
	 Thus we obtain that property $(\mathcal{T})$ holds with constants $\delta$ and $\displaystyle M:=\frac{2}{\kappa}.$
\end{proof}


\section{Primal space characterizations of transversality }

We continue to obtain primal space characterizations of transversality.
A direct consequence of the definition of transversality and Theorem \ref{tansub} is a characterization of transversality in terms of 
``translated'' subtransversality.

\begin{proposition}\label{tsttra}
	Let $A$ and $B$ be closed subsets of the Banach space $X$
	and $\bar x \in A\cap B$. Then $A$ and $B$ are transversal at $\bar x$ if and only if there exist
	$\delta > 0$ and $M > 0$ such that for any $a\in \ball{\delta}{\textbf{0}}$ and $b\in \ball{\delta}{\textbf{0}}$, any $x^A \in A \cap   \ball{\delta}
	{\bar x+a}$ and $x^B \in B \cap    \ball{\delta}{\bar x+b}$ with $x^A-a\not =x^B-b$ there exist $\theta  > 0$, $\hat x^A \in A$ and  $\hat x^B \in B$ such that
	$$\|x^A- \hat x^A\|\le \theta M \ , \ \
	\|x^B- \hat x^B\|\le \theta M \ \mbox{ and } $$
	$$\|\hat x^A-\hat x^B-(a-b)\|\le \|x^A- x^B-(a-b)\|  - \theta  \ .$$
\end{proposition}

\begin{proof}
	Let $A$ and $B$ be transversal at $\bar x$ with constants $K$ and $\hat\delta$. Denote $\delta=\hat\delta/(4K+10)$. Then for all $a\in\ball{\delta}{\mathbf{0}}$ and $b\in\ball{\delta}{\mathbf{0}}$, the sets $A-a$ and $B-b$ have property $(\mathcal{T})$ with constants $\delta$ and $M=K+2$ according to Theorem \ref{tansub}. \\
	Now let the sets satisfy the above property with constants $\delta$ and $M$. Thus for all $\displaystyle a\in\ball{\frac{\delta}{2(1+2M)}}{\mathbf{0}}$ and $\displaystyle b\in\ball{\frac{\delta}{2(1+2M)}}{\mathbf{0}}$, the sets $A-a$ and $B-b$ have property $(\mathcal{T})$ with constants $\delta$ and $M$. Then, again Theorem \ref{tansub} implies that 
	$$d(x,(A-a)\cap(B-b))\le (M+1)(d(x,A-a)+d(x,B-b))$$ 
	for all $x\in\ball{\delta}{\bar x}$, which is precisely transversality.
	
\end{proof}

Strengthening in one of the directions of this proposition gives a characterization of transversality in terms of 
``translated'' tangential transversality.

\begin{proposition}\label{ttttra}
	Let $A$ and $B$ be closed subsets of the Banach space $X$
	and $\bar x \in A\cap B$. Then $A$ and $B$ are transversal at $\bar x$ if and only if there exist
	$\delta > 0$ and $M > 0$ such that for any $a\in \ball{\delta}{\textbf{0}}$ and $b\in \ball{\delta}{\textbf{0}}$, 
	any $x^A \in A \cap   \ball{\delta}	{\bar x+a}$ and $x^B \in B \cap    \ball{\delta}{\bar x+b}$ with $x^A-a\not =x^B-b$ 
	there exist $\{x_m^A\}_{m\ge 1} \subset A$, $\{x_m^B\}_{m\ge 1} \subset B$ and $t_m\searrow 0$ such that  
	$$\|x_m^A-x^A\|\le t_mM \ , \ \
	\|x_m^B-x^B\|\le t_mM \ \mbox{ and } $$
	$$\| x_m^A-x_m^B-(a-b)\|\le \|x^A- x^B-(a-b)\|  - t_m  \ .$$ 
\end{proposition}
\begin{proof}
	The  ``if'' direction is straightforward from Proposition \ref{tsttra}. \\
	For the converse, let $A$ and $B$ be transversal at $\bar x$. This means that the map $H:=H_{A,B}$ from \eqref{hmap} is regular at $((\bar x,\bar x),\textbf{0})$ with constants $\hat\delta$ and $K$. 
	Take $\delta<\hat\delta/4$, $a\in \ball{\delta}{\textbf{0}}$, $b\in \ball{\delta}{\textbf{0}}$ and $x^A\in  A \cap   \ball{\delta}	{\bar x+a}$ and $x^B\in B \cap   \ball{\delta}	{\bar x+b}$. Then
		\begin{align*}
	\|x^A-x^B\|&=\|x^A-(\bar x+a)-(x^B-(\bar x+b))+a-b\|=\\
	&\le \|x^A-(\bar x+a)\|+\|x^B-(\bar x+b)\|+\|a\|+\|b\|\le 4\delta<\hat \delta
	\end{align*}
	Define $\displaystyle v=-\frac{x^A-x^B-(a-b)}{\|x^A-x^B-(a-b)\|}$ and choose a sequence $t_m\searrow 0$ such that
	$$x^A-x^B+t_mv\in \ball{\hat\delta}{\textbf{0}}\, .$$
	Metric regularity of $H$ implies that 
	$$d((x^A,x^B),H^{-1}(x^A-x^B+t_mv))\le K d(H(x^A,x^B),x^A-x^B+t_m
	v).$$
	Since the distance to the empty set is $+\infty$, we have that $H^{-1}(x^A-x^B+t_mv)\ne\varnothing.$
	For $m\ge 1$ consider $(x_m^A,x_m^B)\in H^{-1}(x^A-x^B+t_mv)\subset A\times B$ such that 
	$$\|(x_m^A,x_m^B)-(x^A,x^B)\|\le d((x^A,x^B),H^{-1}(x^A-x^B+t_mv))+t_m.$$
	Metric regularity once again implies 
	\begin{align*}
	\|(x_m^A,x_m^B)-(x^A,x^B)\|&\le d((x^A,x^B),H^{-1}(x^A-x^B+t_mv))+t_m\\
	&\le Kd(x^A-x^B+t_mv,H(x^A,x^B))+t_m\\
	&= K\|x^A-x^B+t_mv-(x^A-x^B)\|+t_m=(K+1)t_m\, .
	\end{align*}
	Finally, we have that
	\begin{align*}
    \|x_m^A-x_m^B-(a-b)\|=&
	\|x^A-x^B+t_m v-(a-b)\|=\\=&\|x^A-x^B-(a-b)\|-t_m\, .
	\end{align*}
	
\end{proof}

\begin{remark}
	In the above proposition we can obtain the (formally) stronger statement that there exists $\lambda>0$ such that the decreasing property holds for any $t\in(0, \lambda]$ instead of the sequence $\{t_n\}_{n=1}^{\infty}$ tending to zero from above.
\end{remark}

\begin{remark}
	Propositions \ref{tsttra} and \ref{ttttra} remain true if we consider translations in only one of the sets,  i.e. we may take $a=0$ and only vary $b$.
\end{remark}

The following proposition is a reformulation of Proposition \ref{tsttra} and Proposition \ref{ttttra} is used for the equality of local and nonlocal slope.
\begin{proposition}\label{prop_couptr}
	Under the assumption that $X$ is a Banach space, $A$ and $B$ are transversal at $\bar x$ if and only if there exist $\delta>0$ and $\kappa>0$ such that for all $a$ and $b$ with 
	$\|a\|\le \delta$ and $\|b\|\le\delta$ and all $x\in (A-a)\cap \ball{\delta}{\bar x}$ and $y\in (B-b)\cap \ball{\delta}{\bar x}$ with $x\ne y$ it holds
	$$|\nabla \phi_{a, b}|^\diamond (x, y) = \sup_{(u,v)\ne (x,y)}\frac{\max\{\phi_{a,b}(x,y)-\phi_{a,b}(u,v),0\}}{\|(x,y)-(u,v)\|}\ge \kappa$$ 
	where $\phi_{a,b}$ denotes the coupling function of $A-a$ and $B-b$ (remind Definition (\ref{coupling}))
	
	Moreover, this is equivalent to the existence of $\hat\kappa>0$ such that for all $x\in (A-a)\cap \ball{\delta}{\bar x}$ and $y\in (B-b)\cap \ball{\delta}{\bar x}$ 
	$$|\nabla \phi_{a, b}| (x, y)=\limsup_{(u,v)\to (x,y)}\frac{\max\{\phi_{a,b}(x,y)-\phi_{a,b}(u,v),0\}}{\|(x,y)-(u,v)\|} \ge\hat\kappa$$ 
\end{proposition}

\begin{proof}
The proof of the first part of the proposition is analogous to the proof Proposition \ref{prop_coupsubtr}.

For the second part, clearly $|\nabla \phi_{a, b}|^\diamond (x, y) \ge |\nabla \phi_{a, b}| (x, y) $, and thus the second slope type property implies the first one.

For the reverse implication, the first part of the Proposition implies that the sets $A$ and $B$ are  transversal at $\bar x \in A\cap B$ with some constants $\delta>0$ and $M>0$. Fix  $a\in \ball{\delta}{\textbf{0}}$ and $b\in \ball{\delta}{\textbf{0}}$ and $x\in (A-a)\cap \ball{\delta}{\bar x}$ and $y\in (B-b)\cap \ball{\delta}{\bar x}$, $x\ne y$. Recall that Proposition  \ref{ttttra} implies that for $x^A:=x+a \in A \cap   \ball{\delta}	{\bar x+a}$ and $x^B:=y+b \in B \cap    \ball{\delta}{\bar x+b}$ (thus $x^A-a\not =x^B-b$) 
	there exist $\{x_m^A\}_{m\ge 1} \subset A$, $\{x_m^B\}_{m\ge 1} \subset B$ and $t_m\searrow 0$ such that  
	$$\|x_m^A-x^A\|\le t_mM \ , \ \
	\|x_m^B-x^B\|\le t_mM \ \mbox{ and } $$
	$$\| x_m^A-x_m^B-(a-b)\|\le \|x^A- x^B-(a-b)\|  - t_m  \ .$$ 
	This is equivalent to
	$$\frac{d(x^A-a,x^B-b)-d(x^A_m-a,x^B_m-b)}{t_m} \ge 1\, . $$
	Using that $d(x^A-a,x^A_m-a)\le Mt_m$ and $d(x^B-b,x^B_m-b)\le Mt_m$, we have that
	$$\frac{d(x^A -a, x^B-b) - d(x^A_m-a , x^B_m-b)}{d(x^A-a,x^A_m-a)+d(x^B-b, x^B_m-b)} \ge \frac{1}{2M} \, .$$
	whence 
	$$\frac{d(x, y) - d(x^A_m-a , x^B_m-b)}{d(x,x^A_m-a)+d(y, x^B_m-b)} \ge \frac{1}{2M} \, .$$
	Since $x\in A-a$ and $y\in B-b$, $\phi_{a,b}(x,y)=d(x,y)$. Similarly $x_m^A-a\in A-a$ and $x_m^B-b\in B-b$, hence $\phi_{a,b}(x_m^A-a,x_m^B-b)=d(x_m^A-a,x_m^B-b).$
	Moreover $x_m^A-a\to x,\ x_m^B-b\to y$. This implies that 
$$|\nabla \phi_{a, b}| (x, y)=\limsup_{(u,v)\to (x,y)}\frac{\max\{\phi_{a,b}(x,y)-\phi_{a,b}(u,v),0\}}{\|(x,y)-(u,v)\|} \ge\frac{1}{2M}$$

\end{proof}

\section{Intrinsic transversality - extensions and related notions}

In this section we provide a metric characterization of intrinsic transversality. This characterization could be used as a definition of intrinsic transversality in general metric spaces. Moreover, we show that it is almost equivalent to the notion of tangential transversality, via observing a slope type characterization of the latter. Finally we show that the metric characterization we provide is equivalent in Hilbert spaces to a characterization introduced and studied in \cite{TBCV2018}.

Fix a point $\bar x\in A\cap B$. First, we provide a characterization of tangential transversality in terms of the slope of the coupling function (remind Definition \ref{coupling}).

\begin{proposition}\label{prop_couptt}
	The subsets $A$ and $B$ of the metric space $X$ are tangentially transversal at $\bar x$ if and only if there exist $\delta>0$ and $\kappa>0$ such that for any two different points 
	$x\in A\cap \ball{\delta}{\bar x}$ and $y\in B\cap \ball{\delta}{\bar x}$ it holds
	\begin{equation*}
	{|\nabla \phi|}(x, y) = \limsup_{(u,v)\to (x,y)}\frac{\max\{\phi(x,y)-\phi(u,v),0\}}{d((x,y),(u,v))}\ge \kappa \, .
	\end{equation*}
\end{proposition} 

\begin{proof}
	The proof is analogous to the proofs of Proposition \ref{prop_coupsubtr} and \ref{prop_couptr}.

\end{proof}

Intrinsic transversality is introduced in \cite{DIL2014}  and \cite{DIL2015} as a sufficient condition for local linear convergence of 
the alternating projections algorithm in finite dimensions. Drusvyatskiy, Ioffe and Lewis found a characterization of intrinsic transversality in finite dimensional spaces in terms of the slope of the coupling function (cf. Proposition 4.2 in \cite{DIL2015}). We use this characterization as a definition of intrinsic transversality in general metric spaces.
\begin{definition}\label{def_intrinsic}
	Let $X$ be a metric space. The closed sets $A, B \subset X$ are intrinsically transversal at the point $\bar x \in A\cap B$, if there
	exist $\delta>0$ and $\kappa>0$ such that for all $x^A\in \ball{\delta}{\bar x} \cap A\setminus B$ and 
	$x^B\in \ball{\delta}{\bar x} \cap B\setminus A$ it holds true that
	$$|\nabla \phi|(x^A, x^B) \ge \kappa \, . $$
\end{definition}

We continue to observe the ``almost'' equivalence of intrinsic transversality and tangential transversality. Due to Proposition \ref{prop_couptt} we have that the only difference between tangential transversality and intrinsic transversality is that in the original definition of tangential transversality the required condition should hold for all points of $A$ and $B$ (respectively) near the reference point, whereas in intrinsic transversality -- only for points in $A\setminus B$ and $B\setminus A$ (respectively).
We introduce the following property.
\begin{definition}[Property $(\mathcal{LT})$]\label{thm_equiv_intr}

We say that the closed sets $A$ and $B$ satisfy property $(\mathcal{LT})$ at $\bar x \in A \cap B$, if there exist $\varepsilon >0$ and  $\theta >0$ such that for any
	two different points $x^A \in   {\overline{\mathbf{B}}}_\varepsilon(\bar x)\cap A\setminus B$ and $x^B \in   {\overline{\mathbf{B}}}_\varepsilon(\bar x)\cap B\setminus A$, there exist sequences $t_m\searrow 0$,  $\{x_m^A\}_{m\ge 1} \subset A$ and $\{x_m^B\}_{m\ge 1} \subset B$ such that for all $m$ $$d(x_m^A, x^A)\le t_m,\ \ d(x_m^B, x^B)\le t_m,\ \ d(x_m^A,x_m^B) \le d(x^A,x^B) - t_m \theta\, .$$
\end{definition}
The comments above yield the following
\begin{corollary}
    The sets $A$ and $B$ are intrinsically transversal at $\bar x\in A\cap B$ if and only if they satisfy property $(\mathcal{LT})$ at $\bar x$.
    
\end{corollary}

In this way we answer a question of Prof. A. Ioffe about finding a metric characterization of intrinsic transversality.

The following example shows that although the difference is slight, the notion of tangential tranvsersality is stronger than the one of intrinsic transversality.
\begin{example} 
	Consider the sets in $\mathbb{R}^2$, $$A=\{(x,y)\ | \ y=3x,\ x\ge 0\}\cup \left\{\left(\frac{1}{n},\frac{2}{n}\right)\right\}_{n\ge 1}$$ and $$B=\{(x,y)\ | \ y=x,\ x\ge 0\}\cup \left\{\left(\frac{1}{n},\frac{2}{n}\right)\right\}_{n\ge 1}.$$
	Apparently these two sets are intrinsically transversal at $(0,0)$, however they are not tangentially transversal, because there are isolated points of the intersection in every neighbourhood of the reference point.
\end{example}
We are also able to answer some of the questions posed in \cite{BKRtt}:
\begin{enumerate}[1.]
	\item \textit{Tangential transversality is an intermediate property between
		transversality and subtransversality.   However, the exact relation between
		this new concept and the established notions of  transversality, intrinsic trans\-ver\-sa\-lity
		and subtransversality is not clarified yet.}
	
	This question is now fully answered in the case of complete metric spaces. The characterizations of intrinsic transversality and tangential transversality show that the examples at the end of Section 6 in \cite{DIL2015} may be used to prove that tangential transversality
	is strictly between transversality and subtransversality even in $\mathbb{R}^d$.
	
	\item \textit{It would be useful    to  find some  dual characterization of
		tangential trans\-ver\-sa\-lity.}
	
	The original definition of intrinsic transversality is stated in dual terms (Definition 2.2 in \cite{DIL2014} and Definition 3.1 in \cite{DIL2015}) -- the closed sets $A, B \subset \mathbb{R}^d$ are intrinsically transversal at the point $\bar x \in A\cap B$, if and  only if there exist $\delta>0$ and $\kappa>0$ such that for all $x^A\in \ball{\delta}{\bar x}\cap A\setminus B$ and $x^B\in\ball{\delta}{\bar x} \cap B\setminus A$ it holds true that
	$$\max\left\{d\left(\frac{x^A-x^B}{\|x^A-x^B\|},\, N_{B}\left(x^B\right)\right), d\left(\frac{x^B-x^A}{\|x^B-x^A\|}, N_{A}\left(x^A\right)\right) \right\} \ge \kappa \, , $$
	where $N_{D}\left(\bar x\right)$ is the proximal or limiting normal cone to $D$ at $\bar x$.		\\
	Replacing ``$x^A\in \ball{\delta}{\bar x}\cap A\setminus B$ and $x^B\in\ball{\delta}{\bar x} \cap B\setminus A$'' with ``$x^A\in \ball{\delta}{\bar x}\cap A$, $x^B\in\ball{\delta}{\bar x} \cap B$ and $x^A\ne x^B$'' we obtain a dual characterization of	tangential trans\-ver\-sa\-lity in finite dimensions.		
\end{enumerate}

It is known that intrinsic transversality and subtransverslity coincide for 
convex sets in finite-dimensional spaces (cf. Proposition 6.1 in \cite{Ioffe} or Corollary 3.4 in \cite{Kruger2017} for an alternative
proof). Both proofs exploit the dual characterizations of
intrinsic transversality and substransversality. Now we can easily obtain the slightly stronger result

\begin{corollary}
	Let $X$ be a Banach space. The closed convex sets $A, B \subset X$ are tangentially transversal at the point 
	$\bar x \in A\cap B$, if and only if they are subtransversal at $\bar x$.
\end{corollary}
\begin{proof}{}
	It is enough to check, that if the sets are subtransversal, they are moreover tangentially transversal  (Definition \ref{tangtr}). According to the primal characterization obtained in Theorem \ref{tansub}, subtransversality implies property $(\mathcal{T})$ with some constants $\delta$ and $M$. Let $x^A\in A\cap\ball{\delta}{\bar x}$ and $x^B\in B\cap\ball{\delta}{\bar x}$. Then there are $\hat x^A \in A$, $\hat x^B\in B$ and $\theta>0$, such that 
	$$\|x^A-\hat x^A\|\le \theta M \ , \ \
	\|x^B- \hat x^B\|\le \theta M \ \mbox{ and } \
	\|\hat x^A- \hat x^B\|\le \|x^A-  x^B\|  - \theta  \ .$$
	Let $\{r_n\}_{n\ge 1}\subset (0,1)$ be a sequence tending to $0$. Since $A$ is convex, \\
	$\displaystyle x_n^A:=\left(1-r_n\right) x^A+r_n\hat x^A\in A$ for all $n\in\mathbb{N}$. Similarly for $x_n^B$. Then 
	\begin{align*}
	&\|x_n^A-x_n^B\|=\left\|\left(1-r_n\right)(x^A-x^B)+r_n(\hat x^A-\hat x^B)\right\|\\
	&\le \left(1-r_n\right)\|x^A-x^B\|+r_n(\|x^A-x^B\|-\theta)=\|x^A-x^B\|-t_n
	\end{align*}
	where $t_n=r_n\theta$.
	Moreover we have
	$$\left\| x_n^A-x^A\right\|= r_n\left\|\hat x^A-x^A\right\|\le r_n \theta M=t_nM,$$ 
	and similarly for $x_n^B-x^B$.
	
\end{proof}

Thus intrinsic transversality also coincides with tangential transversality and subtransversality in the case of convex sets.
This last equivalence is also straight-forward to obtain via function slopes characterizations -- using that for convex functions the limiting slope and the nonlocal slope 
coincide (cf. e.g. Proposition 2.1(vii) in \cite{Kruger2015}), the result follows from Propositions \ref{prop_coupsubtr} and \ref{prop_couptt}. 

In the papers \cite{Kruger2017} and \cite{TBCV2018} a generalization of intrinsic transversality to Hilbert spaces is derived. It is based on the normal structure - Definition 2(ii) in \cite{Kruger2017} and Definition 3 in \cite{TBCV2018}. Moreover, in paper \cite{TBCV2018} a so called property $(\mathcal{P})$ is introduced. It is in primal space terms and is shown to be equivalent to the afformentioned extension of intrinsic transversality in Hilbert spaces based on the normal structure (Definition 2(ii) in \cite{Kruger2017} and Definition 3 in \cite{TBCV2018}).


In order to state it we need the following notation - for a normed space $X$,
$$
\mathrm{d}(A, B, \Omega):=\inf _{x \in \Omega, a \in A, b \in B} \max \{\|x-a\|,\|x-b\|\}, \quad \text { for } A, B, \Omega \subset X
$$
with the convention that the infimum over the empty set equals infinity.

Here is the corresponding definition.
\begin{definition}[Property $(\mathcal{P})$] A pair of closed sets $\{A, B\}$ is said to satisfy property $(\mathcal{P})$ at a point $\bar{x} \in A \cap B$ if there are numbers $\alpha \in( 0,1)$ and $\varepsilon>0$ such that for any $a \in$ $(A \setminus B) \cap\ball{\varepsilon}{\bar x},\ b \in(B \setminus A) \cap\ball{\varepsilon}{\bar x}$ and $x \in \ball{\varepsilon}{\bar x}$ with $\|x-a\|=\|x-b\|$ and number
$\delta>0,$ there exists $\rho \in( 0, \delta)$ satisfying
$$
\mathrm{d}\left(A \cap \ball{\lambda}{a}, B \cap \ball{\lambda}{b}, \ball{\rho}{ x}\right)+\alpha \rho \leq\|x-a\|, \text { where } \lambda:=(\alpha+1 / \sqrt{\varepsilon}) \rho
$$
\end{definition}

It is clearly equivalent to the following sequential form

\begin{definition}[Property $(\mathcal{P'})$] A pair of closed sets $\{A, B\}$ is said to satisfy property $(\mathcal{P'})$ at a point $\bar{x} \in A \cap B$ if there are numbers $\alpha \in( 0,1)$ and $\varepsilon>0$ such that for any $a \in$ $(A \setminus B) \cap\ball{\varepsilon}{\bar x},\ b \in(B \setminus A) \cap\ball{\varepsilon}{\bar x}$ and $x \in \ball{\varepsilon}{\bar x}$ with $\|x-a\|=\|x-b\|$ there exists a sequence $s_n \searrow 0, $ satisfying
$$\mathrm{d}\left(A \cap \ball{\lambda_n}{a}, B \cap \ball{\lambda_n}{b}, \ball{s_n}{ x}\right)+\alpha s_n \leq\|x-a\|, \text { where } \lambda_n:=(\alpha+1 / \sqrt{\varepsilon}) s_n
$$ for large enough $n$.

\end{definition}
The following two theorems show that in general normed spaces property $(\mathcal{P})$ implies property $(\mathcal{LT})$, while in Hilbert spaces they are equivalent.

	\begin{theorem}
		Let $X$ be a normed space, $A$ and $B$ be closed subsets of $X$ and $\bar x\in A\cap B$. Assume that $A$ and $B$ satisfy property $(\mathcal{P})$ at $\bar x$. Then they satisfy property $(\mathcal{LT})$ at $\bar x$.
		\end{theorem}
		
		\begin{proof}
		    
	Let $A$ and $B$ satisfy property $(\mathcal{P'})$ with constants $\varepsilon$ and $\alpha$. Fix any $a \in$ $(A \setminus B) \cap\ball{\varepsilon}{\bar x},\ b \in(B \setminus A) \cap\ball{\varepsilon}{\bar x}$
	and set $\displaystyle x=\frac{a+b}{2}\in\ball{\varepsilon}{\bar x}.$
	Propety $(\mathcal{P'})$ implies that there exists a sequence $s_n\searrow 0$ such that
	$$\mathrm{d}\left(A \cap \ball{\lambda_n}{a}, B \cap \ball{\lambda_n}{b}, \ball{s_n}{ x}\right)+\alpha s_n \leq\frac{1}{2}\|a-b\|, \text { where } \lambda_n:=(\alpha+1 / \sqrt{\varepsilon}) s_n$$
	There exist $a_n\in A \cap \ball{\lambda_n}{a}$, $b_n\in B \cap \ball{\lambda_n}{B}$ and $x_n\in \ball{s_n}{x}$
	such that 
$$\|x_n-a_n\|\le	\mathrm{d}\left(A \cap \ball{\lambda_n}{a}, B \cap \ball{\lambda_n}{b}, \ball{s_n}{ x}\right)+\frac{1}{2}\alpha s_n,$$
$$\|x_n-b_n\|\le	\mathrm{d}\left(A \cap \ball{\lambda_n}{a}, B \cap \ball{\lambda_n}{b}, \ball{s_n}{ x}\right)+\frac{1}{2}\alpha s_n.$$
Summing the latter two inequalities we obtain
$$\|x_n-a_n\|+\|x_n-b_n\|\le	2\mathrm{d}\left(A \cap \ball{\lambda_n}{a}, B \cap \ball{\lambda_n}{b}, \ball{s_n}{ x}\right)+\alpha s_n$$
$$\le \|a-b\|-2\alpha s_n+\alpha s_n=\|a-b\|-\alpha s_n$$
The triangle inequality implies 
$$\|a_n-b_n\|\le \|x_n-a_n\|+\|x_n-b_n\|\le \|a-b\|-\alpha s_n$$
Setting $t_n:=\lambda_n$ we obtain
	
		$$\|a_n-a\|\le t_n,\ \ \|b_n-b\|\le t_n,\ \ \|a_n-b_n\| \le \|a-b\| - t_n \frac{\alpha}{\alpha+1/\sqrt{\varepsilon}}\, .$$
		Thus property $(\mathcal{LT})$ holds with $\varepsilon$ and $\displaystyle\theta:=\frac{\alpha}{\alpha+1/\sqrt{\varepsilon}}.$
		
		\end{proof}
		
		\begin{theorem}
		Let $X$ be a Hilbert space, $A$ and $B$ be closed subsets of $X$ and $\bar x\in A\cap B$. Assume that $A$ and $B$ satisfy property $(\mathcal{LT})$ at $\bar x$. Then they satisfy property $(\mathcal{P})$ at $\bar x$.
		\end{theorem}
		
		\begin{proof}
		Let $A$ and $B$ satisfy property $(\mathcal{LT})$ with constants  $\varepsilon$ and $\theta$. We may assume $\varepsilon<1.$ We shall check that $A$ and $B$ satisfy property $(\mathcal{P'})$ with the same $\varepsilon$ and $\alpha$ to be specified later. To this end fix any $a \in$ $(A \setminus B) \cap\ball{\varepsilon}{\bar x},\ b \in(B \setminus A) \cap\ball{\varepsilon}{\bar x}$ and $x \in \ball{\varepsilon}{\bar x}$ with $\|x-a\|=\|x-b\|$. Denote $v:=x-(a+b)/2$.
	The equation $\|x-a\|=\|x-b\|$ implies $(v,a-b)=0$. Moreover $\|x-a\|=\|(a-b)/2+v\|.$ 
	Denote $\displaystyle\psi=\frac{2\|v\|}{\|a-b\|}.$

	Assume $\displaystyle\psi> \frac{ \theta}{5}$. Set  $\displaystyle s_n=\frac{\|v\|}{n}$. Then putting  $x_n:=\displaystyle\frac{a+b}{2}+\frac{n-1}{n}v $ we obtain $x_n\in \ball{s_n}{x}$, $\|x_n-a\|=\|x_n-b\|$ and
	$$\frac{1}{s_n}(\|x-a\|-\mathrm{d}\left(A \cap \ball{s_n}{a}, B \cap \ball{s_n}{b}, \ball{s_n}{ x}\right))\ge\frac{1}{s_n}(\|x-a\|-\|x_n-a\|)=$$
$$\frac{1}{s_n}\left(\left\|\frac{a-b}{2}+v\right\|-\left\|\frac{a-b}{2}+\frac{n-1}{n}v\right\|\right)=$$
$$\frac{2n-1}{n}\frac{\|v\|}{\left\|\frac{a-b}{2}+v\right\|+\left\|\frac{a-b}{2}+\frac{n-1}{n}v\right\|}\ge\frac{2n-1}{n}\frac{\psi\|v\|}{2\sqrt{\psi^2+1}\|v\|}\ge$$
$$\ge\frac{\psi}{2\sqrt{\psi^2+1}}=:f(\psi).$$
Observe that $f$ is increasing in $[0,\infty)$, so that 
$$\frac{1}{s_n}(\|x-a\|-\mathrm{d}\left(A \cap \ball{s_n}{a}, B \cap \ball{s_n}{b}, \ball{s_n}{ x}\right))\ge f\left(\frac{\theta}{5}\right).$$

Now assume $\displaystyle 0<\psi\le \frac{\theta}{5}$. Then $v\ne 0$. Since $A$ and $B$ satisfy property $(\mathcal{LT})$ there exist sequences $t_n\searrow 0$,  $\{a_n\}_{n\ge 1} \subset (A\setminus B)$ and $\{b_n\}_{n\ge 1} \subset (B\setminus A)$ such that for all $n$ 
		\begin{equation}\label{LT}
		\|a_n-a\|\le t_n,\ \ \|b_n-b\|\le t_n,\ \ \|a_n-b_n\| \le \|a-b\| - t_n \theta\, .
		\end{equation}

	 Denote $$u_n=\frac{a+b}{2}+v-\frac{a_n+b_n}{2}\ \ \text{and}\ \ w_n=\frac{a_n-b_n}{\|a_n-b_n\|}.$$
Observe that 
$$|(u_n,w_n)|\le\left|\left(\frac{a+b-a_n-b_n}{2},\frac{a_n-b_n}{\|a_n-b_n\|}\right)\right|+\left|\frac{1}{\|a_n-b_n\|}(v,(a-a_n)-(b-b_n))\right|$$
$$\le t_n\left(1+\frac{2\|v\|}{\|a_n-b_n\|}\right)\le t_n\left(1+2\psi\right)$$
for large enough $n$. Thus $(u_n,w_n)\to 0$. Clearly $u_n\to v$. Hence for large enough $n$ holds $\|u_n\|^2-(u_n,w_n)\ge \|v\|^2/2$. For these $n$ define
$$\gamma_n:=1-\sqrt{\frac{(1+2\psi)^2 t_n^2-(u_n, w_n)^2}{\|u_n\|^2-(u_n, w_n)^2}}$$
and $$v_n:=\gamma_n(u_n-(u_n,w_n)w_n).$$
Then $\gamma_n\in(0,1)$ for large enough $n$ (and actually $\gamma_n\to 1$ since the numerator in the root tends to $0$, and the denominator stays away from $0$). It is easy to see that  $\left(v_n,w_n\right)=0$ so that $\left(v_n,a_n-b_n\right)=0$. Next we observe that
$$\|v_n-u_n\|^2=\gamma_n^2(\|u_n\|^2-(u_n,w_n)^2)-2\gamma_n(\|u_n\|^2-(u_n,w_n)^2)+\|u_n\|^2=$$
$$=(\gamma_n-1)^2(\|u_n\|^2-(u_n,w_n)^2)+(u_n,w_n)^2=(1+2\psi)^2t_n^2.$$

Thus $\|v_n-u_n\|=(1+2\psi)t_n$. Set $s_n=(1+2\psi)t_n\ge t_n$. Set $x_n:=\displaystyle\frac{a_n+b_n}{2}+v_n$. Since $\|x-x_n\|=\|v_n-u_n\|$, we have $x_n\in \ball{s_n}{x}$ and $\|x_n-a_n\|=\|x_n-b_n\|$. Moreover, using $(\ref{LT})$, $a_n\in A\cap\ball{s_n}{a},\ b_n\in B\cap\ball{s_n}{b}$. Thus
	$$\frac{1}{s_n}(\|x-a\|-\mathrm{d}\left(A \cap \ball{s_n}{a}, B \cap \ball{s_n}{b}, \ball{s_n}{ x}\right))\ge\frac{1}{s_n}(\|x-a\|-\|x_n-a_n\|)=$$
$$\frac{1}{(1+2\psi)t_n}\left(\left\|\frac{a-b}{2}+v\right\|-\left\|\frac{a_n-b_n}{2}+v_n\right\|\right)=$$
$$\frac{1}{(1+2\psi)t_n}\frac{(\left\|\frac{a-b}{2}\right\|-\left\|\frac{a_n-b_n}{2}\right\|)(\left\|\frac{a-b}{2}\right\|+\left\|\frac{a_n-b_n}{2}\right\|)}{\left\|\frac{a-b}{2}+v\right\|+\left\|\frac{a_n-b_n}{2}+v_n\right\|}+\frac{1}{(1+2\psi)t_n}\frac{\|v\|^2-\|v_n\|^2}{\left\|\frac{a-b}{2}+v\right\|+\left\|\frac{a_n-b_n}{2}+v_n\right\|}.$$
For the first summand, using $(\ref{LT}),$ we obtain
$$\frac{1}{(1+2\psi)t_n}\frac{(\left\|\frac{a-b}{2}\right\|-\left\|\frac{a_n-b_n}{2}\right\|)(\left\|\frac{a-b}{2}\right\|+\left\|\frac{a_n-b_n}{2}\right\|)}{\left\|\frac{a-b}{2}+v\right\|+\left\|\frac{a_n-b_n}{2}+v_n\right\|}\ge \frac{\theta}{2(1+2\psi)}\frac{\frac{8}{5}\frac{\|a-b\|}{2}}{\frac{12}{5}\left\|\frac{a-b}{2}+v\right\|}$$
$$\ge\frac{\theta}{3}\frac{1}{(1+2\psi)\sqrt{\psi^2+1}}.$$
For the second summand,  $$\|v\|^2-\|v_n\|^2=(1-\gamma_n^2)\|v\|^2+\gamma_n^2(\|v\|^2-\|u_n\|^2)+\gamma_n^2(u_n, w_n)^2\ge$$
$$\ge\gamma_n^2(\|v\|^2-\|u_n\|^2)=-\gamma_n^2\left(\left(a+b-a_n-b_n,v\right)+\left\|\frac{a+b-a_n-b_n}{2}\right\|^2\right)$$
$$\ge -\gamma_n^2(a+b-a_n-b_n,v)-\gamma_n^2t_n^2\ge-2\gamma_n^2t_n\|v\|-\gamma_n^2t_n^2$$
since $\displaystyle\left\|\frac{a+b-a_n-b_n}{2}\right\|^2\le t_n^2.$
Thus 
$$\frac{1}{(1+2\psi)t_n}\frac{\|v\|^2-\|v_n\|^2}{\left\|\frac{a-b}{2}+v\right\|+\left\|\frac{a_n-b_n}{2}+v_n\right\|}\ge-\frac{\gamma_n^2}{1+2\psi}\left(\frac{2\|v\|}{\left\|\frac{a-b}{2}+v\right\|+\left\|\frac{a_n-b_n}{2}+v_n\right\|}+t_n\right)$$
$$\ge-\frac{\gamma_n^2}{1+2\psi}\left(\frac{2\|v\|}{\frac{8}{5}\left\|\frac{a-b}{2}+v\right\|}+t_n\right)\ge-\frac{4\psi}{3(1+2\psi)\sqrt{\psi^2+1}}.$$
Consequently,
$$\frac{1}{(1+2\psi)t_n}\left(\left\|\frac{a-b}{2}+v\right\|-\left\|\frac{a_n-b_n}{2}+v_n\right\|\right) \ge\frac{\theta-4\psi}{3(1+2\psi)\sqrt{\psi^2+1}}=:g(\psi).$$ 
Observe that $g$ is decreasing in $[0,\infty)$, so that 
$$\frac{1}{s_n}(\|x-a\|-\mathrm{d}\left(A \cap \ball{s_n}{a}, B \cap \ball{s_n}{b}, \ball{s_n}{ x}\right))\ge g\left(\frac{\theta}{5}\right)>0.$$
If $\psi=0$, then $v=0$. Observe that 
$$\left\|\frac{a+b}{2}-\frac{a_n+b_n}{2}\right\|\le t_n.$$
Set  $\displaystyle s_n=t_n$. Then $x_n:=\displaystyle\frac{a_n+b_n}{2}\in \ball{s_n}{x}$, $\|x_n-a_n\|=\|x_n-b_n\|$ and as before $a_n\in A\cap\ball{s_n}{a},\ b_n\in B\cap\ball{s_n}{b}$. Thus 
	$$\frac{1}{s_n}(\|x-a\|-\mathrm{d}\left(A \cap \ball{s_n}{a}, B \cap \ball{s_n}{b}, \ball{s_n}{ x}\right))\ge\frac{1}{s_n}(\|x-a\|-\|x_n-a_n\|)=$$
$$\frac{1}{t_n}\left(\left\|\frac{a+b}{2}-a\right\|-\left\|\frac{a_n+b_n}{2}-a_n\right\|\right)=\frac{\|a-b\|-\|a_n-b_n\|}{2t_n}\ge\frac{\theta}{2}.$$
We conclude 
\begin{equation}\label{last}
    \frac{1}{s_n}(\|x-a\|-\mathrm{d}\left(A \cap \ball{s_n}{a}, B \cap \ball{s_n}{b}, \ball{s_n}{ x}\right))\ge
\end{equation}
$$\min\left\{\frac{\theta}{2},f\left(\frac{\theta}{5}\right),g\left(\frac{\theta}{5}\right)\right\}>0.$$
Set $\displaystyle\alpha:=\min\left\{\frac{\theta}{2},f\left(\frac{\theta}{5}\right),g\left(\frac{\theta}{5}\right)\right\}$ and $\displaystyle\lambda_n:=\left(\alpha+\frac{1}{\sqrt{\varepsilon}}\right)s_n$. Observe that $\lambda_n> s_n$ since $\varepsilon<1.$ Thus, using (\ref{last}) we obtain 
$$  \frac{1}{s_n}(\|x-a\|-\mathrm{d}\left(A \cap \ball{\lambda_n}{a}, B \cap \ball{\lambda_n}{b}, \ball{s_n}{ x}\right))\ge$$
$$  \frac{1}{s_n}(\|x-a\|-\mathrm{d}\left(A \cap \ball{s_n}{a}, B \cap \ball{s_n}{b}, \ball{s_n}{ x}\right))\ge\alpha$$
 
	Finally, property $(\mathcal{P'})$ holds with $\varepsilon$ and $\displaystyle\alpha$.
	\end{proof}

%
%


\begin{thebibliography}{}
%
%
\bibitem{Aze2006} D. Az\'{e}, A unified theory for metric regularity of multifunctions. J. Convex
Analysis 13, 225-252 (2006)

\bibitem{Aze2002} D. Az\'{e}, J.-N. Corvellec, R. E. Lucchetti, Variational pairs and applications to stability in
nonsmooth analysis, Nonlinear Anal., Theory Methods Appl. 49A(5) 49 (2002), 643--670

\bibitem{Aze2004} D. Az\'{e}, J.-N. Corvellec, Characterizations of error bounds for lower semicontinuous func-
tions on metric spaces, ESAIM Control Optim. Calc. Var. 10 (2004), 409--425

\bibitem{BKRtt} M. Bivas, M. Krastanov, N. Ribarska, On tangential transversality, Journal of Mathematical Analysis and Applications, 
Volume 481, Issue 1, 2020, 123445, ISSN 0022-247X, \url{https://doi.org/10.1016/j.jmaa.2019.123445}

\bibitem{GMT1980} E. De Giorgi, A. Marino, and M. Tosques, Problemi di evoluzione in spazi metricie curve
di massima pendenza, Atti Accad. Naz. Lincei Rend. Cl. Sci. Fis. Mat. Natur. 68 (1980), 180--187

\bibitem{Kruger2020} H.T. Bui, N.D. Cuong, A.Y. Kruger, Geometric and Metric Characterizations of Transversality Properties,
\url{http://www.optimization-online.org/DB_HTML/2019/07/7275.html} (2020)

\bibitem{Kruger2019} N. D. Cuong , A.Y. Kruger, Nonlinear transversality of collections of sets: Primal space characterizations,
Preprint, arXiv: 1902.06186 (2019)

\bibitem{nese} N. D. Cuong , Kruger, A.Y.: Primal space necessary characterizations of transversality properties.Preprint Optimization Online 2020-01-7579 (2020)

\bibitem{suff}
Cuong, N.D., Kruger, A.Y.: Transversality properties: primal sufficient conditions. Set-Valued Var.Anal. (2020)
 
 \bibitem{dualsuff} Cuong, N. D., Kruger, A. Y. (2020). Dual sufficient characterizations of transversality properties. Positivity, 1-47.
 
\bibitem{DQZ2006} A. L. Dontchev, M. Quincampoix and N. Zlateva, Aubin criterion for metric regularity,  Journal of Convex Analysis (2006), 13, No 2, 281-297

\bibitem{DonRock} A. L. Dontchev, R. T. Rockafellar, Implicit Functions and Solution Mappings: A View from Variational Analysis,
Springer-Verlag New York (2014)

\bibitem{DIL2014} D. Drusvyatskiy, A.D. Ioffe, and A.S. Lewis, Alternating projections and coupling
slope, \url{http://www.optimization-online.org/DB_HTML/2014/01/4217.html}  (2014)

\bibitem{DIL2015} D. Drusvyatskiy, A.D. Ioffe, and A.S. Lewis, Transversality and Alternating Projections for Nonconvex Sets, 
Found Comput Math 15, 1637–1651 (2015)

\bibitem{Frankowska1990} H. Frankowska, Some inverse mapping theorems, Ann. Inst. H. Poincar\'{e},
Anal. Non Lin\'{e}aire, 7, 1990, 183–234

\bibitem{Ioffe1999} A.D. Ioffe, Variational methods in local and global nonsmooth analysis, In
Clarke F.H., Stern, R.J. (eds.) Nonlinear Analysis, Differential Equations
and Control, NATO Science Series C: Mathematical and Physical Sciences,
vol. 258, pp. 447-502. Kluwer, Dordrecht, Boston, London (1999)

\bibitem{Ioffe2000} A.D. Ioffe, Metric regularity and subdifferential calculus, Russian Math.
Surveys 55(3), 501-558 (2000)

\bibitem{Ioffe} A. Ioffe, Transversality in Variational Analysis, J Optim Theory Appl (2017), 174(2), 343-366

\bibitem{IoffeBook} A. Ioffe, Variational Analysis of Regular Mappings: Theory and Applications, Springer 
Monographs in Mathematics, Springer (2017)

\bibitem{IvanovZlateva2016} M. Ivanov and N. Zlateva, Long Orbit or Empty Value Principle, Fixed Point and Surjectivity Theorems, Compt. rend. Acad. bulg. Sci.,  69, 2016, No 5, 553-562

\bibitem{IvanovZlateva2020} M. Ivanov and N. Zlateva, On Characterizations of Metric Regularity of Multi-Valued Maps, Journal of Convex Analysis 27 (2020), No. 1, 381--388

\bibitem{Jech} T. Jech, Set theory, Academic press (1978)

\bibitem{Kruger2006} A.Y. Kruger, About regularity of collections of sets. Set-Valued Anal. 14, 187–206 (2006)

\bibitem{Kruger2015} A.Y. Kruger, Error bounds and metric subregularity, Optimization (2015), 64:1, 49-79, DOI: 10.1080/02331934.2014.938074

\bibitem{Kruger2017} A.Y. Kruger, About Intrinsic Transversality of Pairs of Sets, Set-Valued Var. Anal 26, 111–142 (2018), https://doi.org/10.1007/s11228-017-0446-3

\bibitem{Kruger2018} A.Y. Kruger, D.R. Luke, N.H. Tao, Set regularities and feasibility prob-lems, Mathematical Programming B (2018), 168, 279–311

\bibitem{TBCV2018} N.H. Thao, H.T. Bui, N.D. Cuong, M. Verhaegen, Some new characterizations of intrinsic 
transversality in Hilbert spaces, Set-Valued and Variational Analysis volume 28, pages 5–39 (2020)
\end{thebibliography}


\end{document}